\documentclass[reqno]{amsart}
\usepackage{graphics}
\usepackage{latexsym}
\usepackage[dvips]{epsfig}
\usepackage{color}
\usepackage{graphicx}
\usepackage{amsmath}
\usepackage{amssymb}
\usepackage{amsfonts}
\usepackage{eufrak}
\usepackage{amstext}
\usepackage{amsopn}
\usepackage{amsbsy}
\usepackage{amscd}
\usepackage{amsxtra}
\usepackage{amsthm}
\usepackage{bm}
\usepackage{CJK}

\newcommand{\R}{{\mathbb R}}

\newcommand{\beq}{\begin{equation}}
\newcommand{\eeq}{\end{equation}}
\newcommand{\ben}{\begin{eqnarray}}
\newcommand{\een}{\end{eqnarray}}
\newcommand{\beno}{\begin{eqnarray*}}
\newcommand{\eeno}{\end{eqnarray*}}

\newtheorem{thm}{Theorem}[section]

\newtheorem{lem}[thm]{Lemma}
\newtheorem{prop}[thm]{Proposition}
\newtheorem{coro}[thm]{Corollary}
\newtheorem{rmk}[thm]{Remark}

\allowdisplaybreaks

\begin{document}

\title[Regularity of flat segregated interfaces]
{Uniform Lipschitz regularity of flat segregated interfaces in a
singularly perturbed problem}
\author[K. Wang]{ Kelei Wang}
 \address{\noindent K. Wang-
 School of Mathematics and Statistics, Wuhan University\\
\& Computational Science Hubei Key Laboratory, Wuhan University,
Wuhan, 430072, China.}
\email{wangkelei@whu.edu.cn}

\begin{abstract}
For the singularly perturbed system
\[\Delta u_{i,\beta}=\beta u_{i,\beta}\sum_{j\neq i}u_{j,\beta}^2, 1\leq i\leq N,\]
we prove that flat interfaces are uniformly Lipschitz. As a
byproduct of the proof we also obtain the optimal lower bound near
the flat interfaces,
\[\sum_iu_{i,\beta}\geq c\beta^{-1/4}.\]
\end{abstract}



\maketitle

\date{}

\section{Main result}

This note is intended as a remark on the recent paper of Soave and
Zilio \cite{SZ 2}. We study the flat segregated interfaces of the
following singularly perturbed elliptic system
\begin{equation}\label{equation scaled}
\Delta u_{i,\beta}=\beta u_{i,\beta}\sum_{j\neq i} u_{j,\beta}^2,
\quad 1\leq i\leq N.
\end{equation}

Assume $u_\beta$ is a sequence of positive solutions to this system
in $B_1(0)\subset\R^n$, satisfying
\[\sup_{B_1(0)}\sum_{i}u_{i,\beta}\leq1, \quad \forall \beta>0.\]
By \cite{SZ}, $u_{i,\beta}$ are uniformly bounded in
$\mbox{Lip}_{loc}(B_1(0))$. Hence we can assume $u_{i,\beta}$
converges to $u_i$ in $C_{loc}(B_1(0))$. (It also converges strongly
in $H^1_{loc}(B_1)$, see \cite{TT2011}.) Then $(u_i)$ satisfies the
segregated condition
\[u_iu_j\equiv 0, \quad \forall \ \  i\neq j.\]
It was proved in \cite{TT2011} (see also \cite{DWZ2011}) that the
free boundary $\cup_i\partial\{u_i>0\}$ has Hausdorff dimension
$n-1$ and it can be decomposed into two parts: $\mbox{Reg}(u_i)$ and
$\mbox{Sing}(u_i)$. $\mbox{Sing}(u_i)$ is a relatively closed subset
of $\cup_i\partial\{u_i>0\}$ of Hausdorff dimension at most $n-2$,
while for any $x\in\mbox{Reg}(u_i)$, there exists a ball $B_r(x)$
such that there are only two components of $(u_i)$ nonvanishing  in
this ball, say $u_1$ and $u_2$. Furthermore, $u_1-u_2$ is harmonic
and $\nabla(u_1-u_2)\neq0$ in this ball. Hence the free boundary in
this ball is exactly the nodal set of this harmonic function. In
\cite{SZ 2}, it was proved that in this ball non-dominating species
decay as follows:
\begin{equation}\label{exponential decay}
\sum_{j\neq 1,2}u_{j,\beta}\leq Ce^{-c\beta^c}.
\end{equation}

Without loss of generality and perhaps after taking a smaller $r$,
we can assume $x=0$ and $\{u_1-u_2=0\}\cap B_r(0)$ is represented by
the graph of a Lipschitz graph in the form $\{x_n=h(x^\prime)\}$,
for $x^\prime\in B_r^{n-1}(0)$.

Our main result is
\begin{thm}\label{main result}
The segregated interface $\{u_{1,\beta}=u_{2,\beta}\}\cap B_r(0)$ is
represented by the graph of a Lipschitz function
$x_n=h_\beta(x_1,\cdots, x_{n-1})$, with the Lipschitz constant of
$h_\beta$ uniformly bounded. Moreover, $h_\beta$ converges uniformly
to $h$ in $B_r^{n-1}(0)$.
\end{thm}

Some corollaries follow from the proof of this theorem.
\begin{coro}\label{coro 1}
There exists a constant $c_1>0$  independent of $\beta$ such that
\begin{equation}\label{1.1}
|\nabla(u_{1,\beta}-u_{2,\beta})|\geq c_1, \quad \mbox{ in } B_r(0).
\end{equation}
\end{coro}

\begin{coro}\label{coro 2}
There exists a constant $c_2>0$ independent of $\beta$ such that,
\[u_{1,\beta}+u_{2,\beta}\geq c_2\beta^{-1/4}, \quad \mbox{ in } B_r(0).\]
\end{coro}
This improves the lower bound estimate in \cite[Theorem 1.6]{SZ 2}
to the optimal one. Corollary \ref{coro 1} is also optimal, in the
sense that there is no further uniform regularity of $\nabla
u_{1,\beta}-\nabla u_{2,\beta}$. For example,
$u_{1,\beta}-u_{2,\beta}$ does not converge to the limit in $C^1$,
see \cite[Proposition 1.16]{SZ 2}.

The argument in this paper is similar to the proof for the
regularity of flat interfaces in the Allen-Cahn equation presented
in the second part of \cite{W3}. The main technical tool is the
improvement of flatness estimate in \cite{W2}. In \cite{W2}, this
estimate is only stated for entire solutions. However,  thanks to
the local uniform Lipschitz estimate in \cite{SZ}, now we can show
that it also holds for local solutions. Several new estimates from
\cite{SZ 2}, especially the exponential decay of non-dominating
species \eqref{exponential decay}, also allows us to treat systems
with more than two equations.

It is natural to conjecture that flat interfaces are also uniformly
bounded in $C^{k,\alpha}$ for any $k\geq1$ and $\alpha\in(0,1)$.
However, this is out of the reach of arguments in this note, which
does not even imply any uniform $C^{1,\alpha}$ regularity. (In the
Allen-Cahn equation, the uniform $C^{1,\alpha}$ regularity is only
achieved by combining this argument with the result in \cite{C-C
3}.)

\section{Proof of main results}

After restricting to a small ball, by a suitable translation and
some scalings, we are in the following setting:
\begin{enumerate}
\item[(1)] $u_\beta$ is a sequence of solutions to \eqref{equation
scaled} in $B_2(0)$;

\item[(2)] $u_\beta$ converges to $u:=(u_1,u_2,0,\cdots,0)$
uniformly in $B_2(0)$, and also strongly in $H^1_{loc}(B_2(0))$;

\item[(3)] $u_{1,\beta}(0)=u_{2,\beta}(0)$;

\item[(4)] there exists a small universal constant $\sigma_0$ (to be determined later) such that, for any $x\in B_1(0)\cap\{u_1-u_2=0\}$,
\begin{equation}\label{close to 1d}
\frac{\int_{B_1(x)}|\nabla u_1|^2+|\nabla u_2|^2}{\int_{\partial
B_1(x)}u_1^2+u_2^2}\leq 1+\sigma_0.
\end{equation}
\end{enumerate}

By multiplying $u_\beta$ and $u$ by a positive constant, we may
assume
\begin{equation}\label{L2 normalized}
\int_{\partial B_1(0)} u_1^2+u_2^2=\int_{\partial B_1}x_n^2.
\end{equation}

Because $u_1(0)-u_2(0)=0$ and $u_1-u_2$ is harmonic, by Almgren
monotonicity formula for harmonic functions, we always have
\[\frac{\int_{B_1(x)}|\nabla u_1|^2+|\nabla u_2|^2}{\int_{\partial
B_1(x)}u_1^2+u_2^2} \geq 1, \quad \forall \ x\in B_1(0)\cap
\{u_1-u_2=0\},\]
 and \eqref{close to 1d} implies the existence of a
unit vector $e$, which we assume to be the $n$-th coordinate
direction, such that
\begin{equation}\label{close to 1d 2}
\sup_{B_1(0)}\left(|u_1-u_2-x_n|+|\nabla(u_1-u_2-x_n)|\right)\leq
c(\sigma_0)<1/16,
\end{equation}
provided $\sigma_0$ has been chosen small enough.

Some remarks are in order.
\begin{rmk}
In the following it is always assumed that \eqref{exponential decay}
holds in $B_2(0)$. Then because $u_{i,\beta}$ is nonnegative and
subharmonic, we get
\[\sum_{i\neq 1,2}\int_{B_{3/2}(0)}|\nabla u_{i,\beta}|^2\leq Ce^{-c\beta^c}.\]

The following rescaling will be used many times in the proof:
\begin{equation}\label{rescaling}
u^\lambda_{i,\beta}(x)=\lambda^{-1}u_{i,\beta}(\lambda x), \quad
\lambda>0.
\end{equation}
 Once
$\lambda>\beta^{-1/4}$, \eqref{exponential decay} still holds for
$u_\beta^\lambda:=(u_{i,\beta}^\lambda)$, perhaps with a larger $C$
and a smaller $c$ (but still independent of $\beta\to+\infty$).
\end{rmk}

\begin{rmk}\label{uniform Lip bound}
Throughout this section, we assume the Lipschitz constant of
$u_{i,\beta}$ is bounded by a constant independent of $\beta$. Since
all of the rescalings used in this paper are in the form
\eqref{rescaling}, any rescaling of $u_\beta$ has the same Lipschitz
bound.
\end{rmk}

Let us first recall some known results. The first one is the Almgren
monotonicity formula, see for example \cite[Proposition 5.2]{BTWW}.
\begin{prop}\label{Almgren monotonicity formula}
For any $x\in B_2(0)$,
\[N(r;x,u_\beta):=\frac{r\int_{B_r(x)}\sum_i|\nabla u_{i,\beta}|^2+\beta\sum_{i<j}u_{i,\beta}^2u_{j,\beta}^2}
{\int_{\partial B_r(x)}\sum_i u_{i,\beta}^2}\] is increasing in
$r\in(0,2-|x|)$.
\end{prop}
By the strong convergence of $u_\beta$ in $H^1_{loc}(B_2(0))$ and
the bound \eqref{close to 1d}, we can assume that, for all $\beta$
large and $x\in\{u_{1,\beta}=u_{2,\beta}\}\cap B_1(0)$,
$N(1;x,u_\beta)\leq 1+2\sigma_0$. Then by this proposition,
\begin{equation}\label{bound on Almgren}
N(r;x,u_\beta)\leq 1+2\sigma_0, \quad \forall \  r\in(0,1).
\end{equation}

The next one is \cite[Lemma 6.1]{W} or \cite[Theorem 1.1]{SZ 2}.
\begin{lem}\label{lem 2.4}
For any $x\in\{u_{1,\beta}=u_{2,\beta}\}\cap B_{3/2}(0)$,
\[u_{1,\beta}(x)=u_{2,\beta}(x)\leq C\beta^{-1/4}.\]
\end{lem}

The main technical result we will use is the following decay
estimate, first proved in \cite{W2}.
\begin{thm}\label{thm decay estimate}
There exist four universal constants $\theta\in(0,1/2)$,
$\varepsilon_0$ small and $K_0, C$ large such that, if $u_\beta$ is
a solution of \eqref{equation scaled} in $B_1(0)$, satisfying
\begin{equation}\label{exponential decay 2}
\sum_{i\neq
1,2}\left[\sup_{B_1(0)}u_{i,\beta}^2+\int_{B_1(0)}|\nabla
u_{i,\beta}|^2\right]\leq Ce^{-c\beta^c},
\end{equation}
\begin{equation}\label{small condition}
\varepsilon^2:=\int_{B_1(0)}|\nabla u_{1,\beta}-\nabla
u_{2,\beta}-e|^2\leq \varepsilon_0^2,
\end{equation}
where $e$ is a vector satisfying $|e|\geq 1/4$, and
$\beta^{1/8}\varepsilon^2\geq K_0$, then there exists another vector
$\tilde{e}$, with
\[|\tilde{e}-e|\leq C(n)\varepsilon,\]
such that
\[\theta^{-n}\int_{B_\theta(0)}|\nabla u_{1,\beta}-\nabla u_{2,\beta}-\tilde{e}|^2\leq \frac{1}{2}\varepsilon^2.\]
\end{thm}
\begin{proof}
The proof is similar to \cite[Theorem 2.2]{W2} with only three
different points:
\begin{enumerate}
\item[(i)] Now the system \eqref{equation scaled} could contain
more than two equations. However, with the hypothesis
\eqref{exponential decay 2} the effect of $u_{i,\beta}$ ($i\neq
1,2$) is exponentially small, hence it does not affect the final
conclusion.

\item[(ii)] We do not claim \cite[Lemma 3.4]{W2}. This estimate is used in \cite[Eq. (5.1)]{W2}. Instead, we only
provide a weaker estimate
\begin{equation}\label{4.1}
\int_{B_{3/4}(0)}\beta u_{1,\beta}u_{2,\beta}^3+\beta
u_{2,\beta}u_{1,\beta}^3\leq C\beta^{-1/8}.
\end{equation}
This is the reason we replace the condition
$\varepsilon^2\gg\beta^{-1/4}$ in \cite[Theorem 2.2]{W2} by a more
restrictive one $\varepsilon^2\gg\beta^{-1/8}$.

Note that $\beta u_{1,\beta}u_{2,\beta}^3\leq u_{2,\beta}\Delta
u_{1,\beta}$. Thus
\begin{eqnarray*}
\int_{B_{3/4}(0)}\beta
u_{1,\beta}u_{2,\beta}^3&\leq&\int_0^{+\infty}\left(\int_{B_{3/4}(0)\cap\{u_{2,\beta}>t\}}\Delta
u_{1,\beta}\right)dt\\
&\leq&\int_0^{L\beta^{-1/4}\log\beta}\int_{B_{3/4}(0)}\Delta
u_{1,\beta}+\int_{L\beta^{-1/4}\log\beta}^{+\infty}Ce^{-c\beta^{1/4}t}\\
&\leq&C\beta^{-1/4}\log\beta,
\end{eqnarray*}
where $L$ is a large constant (fixed to be independent of $\beta>0$)
and we have used the fact that $\Delta u_{1,\beta}\leq
Ce^{-c\beta^{1/4}t}$ in $\{u_{2,\beta}>t\}$. \eqref{4.1} follows
from this estimate if $\beta$ is large enough.

\item[(iii)] It is also not known whether \cite[Lemma 3.3]{W2} holds. However, in \cite{W2} this estimate
is only used to derive \cite[Eq. (4.6)]{W2}, which will be replaced
by the following weaker estimate
\begin{equation}\label{4.6}
\int_{B_{3/4}(0)}|\nabla u_{1,\beta}||\nabla u_{2,\beta}|\leq
C\beta^{-1/8}.
\end{equation}

For simplicity, we will take a rescaling as in \eqref{rescaling} so
that $\beta=1$ in the equation and the domain is $B_R(0)$ where
$R=\beta^{1/4}$. Solutions are denoted by $(u_i)$.

Choose a $T$ large so that $u_1u_2<T^2$ in $B_R(0)$ (see \cite[Lemma
6.1]{W}). By this choice $\{u_1>T\}$ and $\{u_2>T\}$ are disjoint.

For any $x\in\{u_1<T,u_2<T\}$, by the Lipschitz continuity of $u_1$
and $u_2$, $u_1\leq T+C$ and $u_2\leq T+C$ in $B_1(x)$. Then by
standard gradient estimates and Harnack inequality,
\[|\nabla u_i(x)|\leq C\sup_{B_1(x)}u_i\leq Cu_i(x), \quad\forall \  i=1,2.\]
Thus by the Cauchy inequality,
\begin{eqnarray}\label{4.3}
\int_{B_{R-1}(0)\cap\{u_1<T,u_2<T\}}|\nabla u_1||\nabla u_2|&\leq&
C\int_{B_{R-1}(0)\cap\{u_1<T,u_2<T\}} u_1u_2 \nonumber\\
&\leq&CR^{\frac{n}{2}}\left(\int_{B_{R-1}(0)}u_1^2u_2^2\right)^{1/2}\\
&\leq& CR^{n-1/2},\nonumber
\end{eqnarray}
where we have used \cite[Lemma 6.4]{W}, which implies
\[\int_{B_{R-1}(0)}u_1^2u_2^2\leq CR^{n-1}.\]

For $x\in\{u_1\geq T\}$, by noting that
\[\Delta |\nabla u_2|\geq u_1^2|\nabla u_2|-2u_1u_2|\nabla u_1|,\]
we get
\begin{equation}\label{4.5}
|\nabla u_2(x)|\leq C\sup_{B_{1/2}(x)}\left(u_1u_2\right).
\end{equation}
 Because $u_2$ is
subharmonic,
\begin{equation}\label{4.4}
\sup_{B_{1/2}(x)}u_2\leq C\int_{B_1(x)}u_2.
\end{equation}
Since $u_1(x)\geq T$, by the Lipschitz bound on $u_1$, if we have
chosen $T$ sufficiently large,
\begin{equation}\label{4.7}
\frac{1}{2}\sup_{B_1(x)}u_1\leq u_1(x)\leq \sup_{B_1(x)}u_1.
\end{equation}
Combining \eqref{4.5}-\eqref{4.7} with the Lipschitz continuity of
$u_1$, we get
\[|\nabla u_1(x)||\nabla u_2(x)|\leq C\int_{B_1(x)}u_1u_2, \quad \forall \ x\in\{u_1>T\}\cap B_{3R/4}(0).\]
Integrating this on $\{u_1>T\}\cap B_{3R/4}(0)$ and using the Fubini
theorem and the Cauchy inequality, we obtain
\begin{eqnarray}\label{4.8}
\int_{\{u_1>T\}\cap B_{3R/4}(0)}|\nabla u_1||\nabla u_2|&\leq&
C\int_{B_1(0)}\int_{\{u_1>T\}\cap B_{3R/4}(0)}u_1(x+y)u_2(x+y)dxdy \nonumber\\
&\leq& C\int_{B_{\frac{3R}{4}+1}(0)}u_1u_2\\
&\leq& CR^{n-1/2}.\nonumber
\end{eqnarray}
A similar estimate holds in $\{u_2>T\}\cap B_{3R/4}(0)$. Combining
\eqref{4.3} with these we get \eqref{4.6}.
\end{enumerate}
\end{proof}

The next lemma can be used to show that the condition \eqref{small
condition} is always satisfied for $(u_{i,\beta}^\lambda)$, provided
$\lambda\gg \beta^{-1/4}$.
\begin{lem}\label{lem excess small}
For any $\varepsilon>0$, there exist two constants $K(\varepsilon)$
and $\delta(\varepsilon)$ so that the following holds. Suppose
$u_\beta$ is a solution of \eqref{equation scaled} in $B_2(0)$, with
$\beta\geq K(\varepsilon)$, satisfying
$u_{1,\beta}(0)=u_{2,\beta}(0)$, \eqref{exponential decay 2} and
\begin{equation}\label{Almgren bound}
\frac{2\int_{B_2(0)}\sum_{i}|\nabla u_{i,\beta}|^2+\sum_{i<j}\beta
u_{i,\beta}^2u_{j,\beta}^2} {\int_{\partial B_2(0)}\sum_i
u_{i,\beta}^2}\leq 1+\delta(\varepsilon),
\end{equation}
 then there exists a vector
$e$ such that
\begin{equation}\label{excess small}
\int_{B_1(0)}|\nabla u_{1,\beta}-\nabla
u_{2,\beta}-e|^2\leq\varepsilon^2.
\end{equation}
\end{lem}
\begin{proof}
Assume by the contrary, there exists an $\varepsilon>0$, a sequence
of solutions $u_\beta$ with $\beta\to+\infty$, satisfying
$u_{1,\beta}(0)=u_{2,\beta}(0)$, \eqref{exponential decay 2} and
\begin{equation}\label{Almgren bound 2}
\limsup_{\beta\to+\infty}\frac{2\int_{B_2(0)}\sum_{i}|\nabla
u_{i,\beta}|^2+\sum_{i<j}\beta u_{i,\beta}^2u_{j,\beta}^2}
{\int_{\partial B_2(0)}\sum_i u_{i,\beta}^2}\leq 1,
\end{equation}
but for any vector $e$,
\begin{equation}\label{excess small 2}
\int_{B_1(0)}|\nabla u_{1,\beta}-\nabla
u_{2,\beta}-e|^2\geq\varepsilon^2.
\end{equation}

By our assumption, the Lipschitz constant of $u_{i,\beta}$ in
$B_{3/2}(0)$ are uniformly bounded in $\beta$. By Lemma \ref{lem
2.4},
\[u_{1,\beta}(0)=u_{2,\beta}(0)\leq C\beta^{-1/4}.\]
Hence $u_{1,\beta}$ and $u_{2,\beta}$ are also uniformly bounded in
$B_{3/2}(0)$. Assume it converges uniformly to $(u_1,u_2,0,\cdots)$
in $B_{3/2}(0)$. As before, $u_1u_2\equiv0$ and $u_1-u_2$ is a
harmonic function. Moreover, $(u_{i,\beta})$ also converges to
$(u_1,u_2,0,\cdots)$ in $H^1(B_1(0))$. Hence by Proposition
\ref{Almgren monotonicity formula} and \eqref{Almgren bound 2}, we
obtain
\[\frac{\int_{B_1(0)}\sum_{i}|\nabla
u_i|^2}{\int_{\partial B_1(0)}\sum_i u_i^2}\leq 1.\] Then by  the
characterization of linear functions using Almgren monotonicity
formula (noting that $u_1(0)-u_2(0)=0$), we get a vector $e$ such
that
\[u_1(x)-u_2(x)\equiv e\cdot x, \quad \mbox{in  } B_1(0).\]
By the strong convergence of $u_{i,\beta}$ in $H^1(B_1(0))$ again,
\[\lim_{\beta\to+\infty}\int_{B_1(0)}|\nabla u_{1,\beta}-\nabla u_1|^2
+|\nabla u_{2,\beta}-\nabla u_2|^2=0.\] This is a contradiction with
\eqref{excess small 2} and finishes the proof of this lemma.
\end{proof}

After these preliminaries now we prove
\begin{lem}\label{lem 2.7}
For any $\sigma>0$, there exist two universal constants
$K_1(\sigma), K_2$ ($K_2$ independent of $\sigma$) such that the
following holds. For any $x\in\{u_{1,\beta}=u_{2,\beta}\}\cap
B_1(0)$, there exists an $r_\beta(x)\in (K_1\beta^{-1/4},\theta)$
such that,
\begin{itemize}
\item for any $r>r_\beta(x)$, there exists a vector $e(r,x)$, with
$|e(r,x)|\geq 1/4$, such that
\begin{equation}\label{Morrey bound 1}
r^{-n}\int_{B_r(x)}|\nabla u_{1,\beta}-\nabla
u_{2,\beta}-e(r,x)|^2\leq Cr^\alpha\int_{B_{3/2}(0)}|\nabla
u_{1,\beta}-\nabla u_{2,\beta}-e_n|^2,
\end{equation}
 where $\alpha=\log
2/|\log\theta|$ and $\theta$ is as in Theorem \ref{thm decay
estimate};

\item for $r\in(K_1\beta^{-1/4},r_\beta(x))$, there exists a vector
$e(r,x)$, with $|e(r,x)|\geq 1/4$, such that
\begin{equation}\label{Morrey bound 2}
r^{-n}\int_{B_r(x)}|\nabla u_{1,\beta}-\nabla
u_{2,\beta}-e(r,x)|^2\leq K_2\beta^{-\frac{1}{8}}r^{-\frac{1}{2}}.
\end{equation}
\end{itemize}
Moreover, for any $r\in(K_1\beta^{-1/4},\theta)$,
\begin{equation}\label{difference with e_n}
|e(r,x)-e_n|\leq \sigma+ C\left(\int_{B_{3/2}(0)}|\nabla
u_{1,\beta}-\nabla u_{2,\beta}-e_n|^2\right)^{1/2}<1/2,
\end{equation}
for all $\beta$ large.
\end{lem}
\begin{proof}
Without loss of generality assume $x$ is the origin $0$. For each
$k\geq 0$, let
\[E_k:=\min_{e\in\R^n} \theta^{-kn}\int_{B_{\theta^k}(0)}|\nabla u_{1,\beta}-\nabla u_{2,\beta}-e|^2,\]
which can be assumed to be attained by a vector $e_k$.

By our hypothesis, in particular \eqref{close to 1d 2}, $E_0$ is
very small for all $\beta$ large. Moreover, $e_0$ is close to the
$n$-th direction. In the following we will show that $|e_k|\geq 1/2$
up to scales $\theta^k\sim \beta^{-1/4}$.

{\bf Claim 1.} For any $k\geq0$, $E_k\geq\theta^n E_{k+1}$.\\
This is because, for any vector $e$,
\[\theta^{-kn}\int_{B_{\theta^k}(0)}|\nabla u_{1,\beta}-\nabla u_{2,\beta}-e|^2
\geq \theta^{-kn}\int_{B_{\theta^{k+1}}(0)}|\nabla
u_{1,\beta}-\nabla u_{2,\beta}-e|^2.\]

Let $\varepsilon_0$ be as in Theorem \ref{thm decay estimate}. Then
 choose $\sigma_0$ and $\tilde{K}_1$ according to Lemma \ref{lem excess small} so that
$2\sigma_0\leq\delta(\varepsilon_0)$ and $\tilde{K}_1\geq
K(\varepsilon_0)$. By Lemma \ref{lem excess small}, we obtain

{\bf Claim 2.} If $\beta^{1/4}\theta^k\geq \tilde{K}_1$, then
$E_k\leq \varepsilon_0^2$.

In the following we take $\tilde{k}_1$ to be the largest $k$
satisfying $\beta^{1/4}\theta^k\geq \tilde{K}_1$. $k_1$ is defined
to be the largest $k\leq \tilde{k}_1$ so that for any $i\leq k$,
$|e_i|\geq 1/2$.

{\bf Claim 3.} For any $1\leq k\leq k_1$, if $E_k\geq
K_2\beta^{-1/8}\theta^{-k/2}$, where $K_2=K_0\theta^{-n}$, then
$E_k\leq \frac{1}{2}E_{k-1}$.

 Let
\[\tilde{u}_{i,\beta}(x):=\theta^{1-k}u_{i,\beta}(\theta^{k-1}x),\]
which satisfies \eqref{equation scaled} with $\beta$ replaced by
$\beta_{k-1}:=\beta\theta^{4k-4}$.

By Claim 2,
\[\varepsilon_{k-1}^2:=\int_{B_1(0)}|\nabla \tilde{u}_{1,\beta}-\nabla
\tilde{u}_{2,\beta}-e_{k-1}|^2=E_{k-1}\leq \varepsilon_0^2.\]
 By Claim 1,
$E_{k-1}\geq K_0\beta_{k-1}^{-1/8}$. Thus
$\beta_{k-1}^{1/8}\varepsilon_{k-1}^2\geq K_0$. Moreover, by
definition we also have $|e_{k-1}|\geq 1/2$. Hence Theorem \ref{thm
decay estimate} applies, which implies the existence of a vector
$\tilde{e}_k$ such that
\[\theta^{-n}\int_{B_\theta(0)}|\nabla \tilde{u}_{1,\beta}-\nabla \tilde{u}_{2,\beta}-\tilde{e}_k|^2
\leq \frac{1}{2}\varepsilon_k^2.\]
 Rescaling back, by the definition of $E_k$, we get {\bf Claim 3}.

Note that in Claim 3, trivially we also have $E_{k-1}\geq E_k$. Thus
we still have
\[E_{k-1}\geq K_2\beta^{-\frac{1}{8}}\theta^{-\frac{k}{2}}\geq
K_2\beta^{-\frac{1}{8}}\theta^{\frac{1-k}{2}}.\] Hence Claim 3 can
be applied repeatedly. From this we deduce the existence of a $k_2$
such that, for any $k\geq k_2$, $E_k\leq
K_2\beta^{-1/8}\theta^{-k/2}$, while for any $k\leq k_2$, $E_k\geq
K_2\beta^{-1/8}\theta^{-k/2}$, and hence by Claim 3,
\[E_k\leq 2^{-1}E_{k-1}\leq \cdots\leq 2^{-k}E_0.\]

It remains to show that $\theta^{k_1}\sim\beta^{-1/4}$. For $k\leq
k_2$, because
\begin{eqnarray*}
\theta^{-kn}\int_{B_{\theta^k}(0)}|e_k-e_{k-1}|^2&\leq&2\theta^{-kn}\int_{B_{\theta^k}(0)}|\nabla
u_{1,\beta}-\nabla
u_{2,\beta}-e_k|^2\\
&&+2\theta^{-kn}\int_{B_{\theta^k}(0)}|\nabla
u_{1,\beta}-\nabla u_{2,\beta}-e_{k-1}|^2\\
&\leq&2E_k+2\theta^{-n}E_{k-1}\\
&\leq&CE_0 2^{-k},
\end{eqnarray*}
we get
\begin{equation}\label{1}
|e_k-e_{k-1}|\leq CE_0^{\frac{1}{2}}2^{-\frac{k}{2}}.
\end{equation}
Similarly, for $k\geq k_2$,
\begin{equation}\label{2}
|e_k-e_{k+1}|\leq
C(n)K_2^{\frac{1}{2}}\beta^{-\frac{1}{16}}\theta^{-\frac{n+k}{4}}.
\end{equation}

Let $k_3$ be the largest number satisfying
\begin{equation}\label{smallest scale}
C(n)K_2^{\frac{1}{2}}\beta^{-\frac{1}{16}}\theta^{-\frac{n}{4}}\frac{\theta^{-\frac{k+1}{4}}}{\theta^{-1/4}-1}\leq\sigma.
\end{equation}
Note that by this choice, there exists a universal constant $C$ such
that
\begin{equation}\label{4}
\frac{1}{C\sigma}\beta^{-\frac{1}{4}}\leq\theta^{k_3}\leq
\frac{C}{\sigma}\beta^{-\frac{1}{4}}.
\end{equation}

Adding \eqref{1} and \eqref{2} from $k=0$ to $k$, we see for any
$k\leq k_3$,
\begin{equation}\label{3}
|e_k-e_0|\leq CE_0^{\frac{1}{2}}+\sigma<1/4.
\end{equation}
In particular, $|e_k|\geq 1/2$ for all $k\leq k_3$. Thus we can
choose $k_1\geq k_3$. By \eqref{4},
\[\theta^{k_1}\leq \frac{C}{\sigma}\beta^{-\frac{1}{4}}.\]

Finally, by choosing $K_1:=\max\{\tilde{K}_1,
\theta^{k_3}\beta^{1/4}\}$ and $r_\beta:=\theta^{k_2}$ we finish the
proof.
\end{proof}

\begin{lem}\label{lem 2.8}
For any $\varepsilon>0$, there exists two constant
$\widetilde{\delta}(\varepsilon)$ and $\tilde{K}(\varepsilon)$ so
that the following holds. Let $u_\beta$ be a solution of
\eqref{equation scaled} in $B_2(0)$ with $\beta\geq
\tilde{K}(\varepsilon)$, satisfying $u_{1,\beta}(0)=u_{2,\beta}(0)$,
\eqref{exponential decay 2} and
\begin{equation}\label{2.1}
\int_{B_2(0)}|\nabla u_{1,\beta}-\nabla u_{2,\beta}-e|^2\leq
\widetilde{\delta}(\varepsilon)
\end{equation}
 for some vector $e$ with $|e|\geq 1/4$. Then
$\{u_{1,\beta}=u_{2,\beta}\}\cap B_1(0)$ belongs to the
$\varepsilon$ neighborhood of $P_e\cap B_1(0)$, where $P_e$ is the
hyperplane orthogonal to $e$.
\end{lem}
\begin{proof}
Assume by the contrary, there exists an $\varepsilon>0$ and a
sequence of solutions $u_\beta$ in $B_2(0)$, with $\beta\to+\infty$,
satisfying $u_{1,\beta}(0)=u_{2,\beta}(0)$, \eqref{exponential decay
2} and
\begin{equation}\label{2.2}
\lim_{\beta\to+\infty}\int_{B_2(0)}|\nabla u_{1,\beta}-\nabla
u_{2,\beta}-e|^2=0,
\end{equation}
where $|e|\geq 1/4$. (At first this vector may depend on $\beta$,
but we can rotate $(u_{i,\beta})$ to make it the same one.) But
there exists $x_\beta\in B_1(0)\cap\{u_{1,\beta}=u_{2,\beta}\}$ such
that
\begin{equation}\label{absurd assump}
\liminf_{\beta\to+\infty}\mbox{dist}(x_\beta, P_e)>0.
\end{equation}
 Hence we can
assume $x_\beta\to x_\infty$, which lies outside $P_e$.

By these assumptions and the uniform Lipschitz regularity of
$u_\beta$, they are uniformly bounded in $\mbox{Lip}_{loc}(B_2(0))$
and can be assumed to converge to a limit $(u_i)$ in
$C_{loc}(B_2(0))$. By \eqref{exponential decay 2}, $u_i\equiv 0$ for
all $i\neq 1,2$. By \eqref{2.2},
\begin{equation}
\int_{B_2(0)}|\nabla u_1-\nabla u_2-e|^2=0.
\end{equation}
Hence by the main result in \cite{TT2011} and \cite{DWZ2011},
$u_1=(e\cdot x)^+$ and $u_2=(e\cdot x)^-$.

Because $u_{i,\beta}\to u_i$ uniformly in $\overline{B_1}$, by the
nondegeneracy of $u_1-u_2$, we obtain a contradiction with
\eqref{absurd assump}.
\end{proof}

Fix an $\varepsilon>0$ and then choose a sufficiently small
$\sigma\leq\tilde{\delta}(\varepsilon)/2$ and a sufficiently large
$K_3\geq \tilde{K}(\varepsilon)$ according to this lemma.
 By Lemma \ref{lem 2.7}, Lemma \ref{lem 2.8} applies to $u_\beta$ in $B_r(x)$ for $r\geq K_3\beta^{-1/4}$
(after scaling to the unit ball), which says
$\{u_{1,\beta}=u_{2,\beta}\}\cap B_r(x)$ belongs to the $\varepsilon
r$ neighborhood of $(x+P_{e(r,x)})\cap B_r(x)$. Since
$|e(r,x)-e_n|\leq 2\sigma$ (for $\beta$ sufficiently large and $e_n$
denotes the $n$-th direction), this implies
$\{u_{1,\beta}=u_{2,\beta}\}\cap B_1(x)\subset
\{|\Pi_{e_n}(y-x)|\leq C\sigma|\Pi_{e_n}^\perp (y-x)|\}$ once
$|y-x|\geq K_3\beta^{1/4}$. Roughly speaking, this is equivalent to
saying that $\{u_{1,\beta}=u_{2,\beta}\}$ is Lipschitz up to the
scale $K_3\beta^{-1/4}$ in the direction $e_n$.

The next result shows that this Lipschitz property also holds for
$r\in(0,K_3\beta^{-1/4})$.
\begin{lem}\label{lem 2.9}
For any $\delta>0$ (sufficiently small) and $L>0$, there exists an
$R(\delta,L)$ so that the following holds. Suppose $(u_i)$ is a
solution of \eqref{equation scaled} with $\beta=1$, in a ball
$B_R(0)$ with $R\geq R(\delta,L)$, satisfying $u_1(0)=u_2(0)$,
\begin{equation}\label{3.1}
\sup_{B_L(0)}\sum_{i\neq 1,2}u_i\leq Ce^{-cR^c},
\end{equation}
 and
 \begin{equation}\label{3.2}
r^{-n}\int_{B_r(0)}|\nabla u_1-\nabla u_2-e|^2\leq\delta, \quad
\forall \ L<r<R,
\end{equation}
where $e$ is a unit vector.
 Then
\begin{equation}\label{3.3}
\sup_{B_L(0)}|\nabla u_1-\nabla u_2-e|\leq c(n)<1.
\end{equation}
Moreover, $\{u_1=u_2\}\cap B_L(0)$ is a Lipschitz graph in the
direction $e$, with its Lipschitz constant bounded by
$\bar{c}(\delta)$, which satisfies
$\lim_{\delta\to0}\bar{c}(\delta)=0$.
\end{lem}
\begin{proof}
Assume by the contrary, there exist $\delta$ and $L$, and a sequence
of solutions $(u_{i,R})$ defined in $B_R(0)$ with $R\to+\infty$,
satisfying \eqref{3.1} and \eqref{3.2}, but the conclusion of this
lemma does not hold.

Because $u_{1,R}(0)=u_{2,R}(0)$, by the Lipschitz bound, there
exists a universal constant $C$ such that
\[u_{1,R}=u_{2,R}(0)\leq C.\]
Combining this with \eqref{3.1} and the uniform Lipschitz bound on
$u_{i,R}$, we see $(u_{i,R})$ are uniformly bounded in
$\mbox{Lip}_{loc}(\R^n)$. Then using standard elliptic estimates and
compactness results, we deduce that $(u_{i,R})$ converges to a limit
$(u_i)$ in $C^2_{loc}(\R^n)$, which is a solution of \eqref{equation
scaled} with $\beta=1$ in $\R^n$.

Passing to the limit in \eqref{3.1} gives $u_i(0)=0$ for all $i\neq
1,2$. Since $u_i\geq0$, by the strong maximum principle, $u_i\equiv
0$ for all $i\neq 1,2$. \eqref{3.2} can also be passed to the limit,
which gives
\begin{equation}\label{3.4}
r^{-n}\int_{B_r(0)}|\nabla u_1-\nabla u_2-e|^2\leq\delta, \quad
\forall \ r>L.
\end{equation}
In particular, because $e$ is nonzero, $(u_1,u_2)\neq 0$ .

It is clear that $(u_1,u_2)$ is a globally Lipschitz solution of the
system
\begin{equation}\label{two component system}
\Delta u_1=u_1u_2^2, \quad \Delta u_2=u_2u_1^2, \quad \mbox{in }
\R^n.
\end{equation}
 Then the main result in \cite{W} says
$(u_1,u_2)=(g_1(\tilde{e}\cdot x), g_2(\tilde{e}\cdot x))$, where
$\tilde{e}$ is a vector and $(g_1,g_2)$ is the one dimensional
solution of \eqref{two component system}. (It is essentially unique,
see \cite{blwz} and \cite{BTWW}.) Substituting this into \eqref{3.4}
we get
\[|\tilde{e}-e|\leq C\delta<1/16,\]
provided $\delta$ has been chosen small enough. (Note that
\eqref{two component system} has a scaling invariance, which however
is fixed by the condition \eqref{3.4}.)

By the implicit function theorem, for all $R$ large,
$\{u_{1,R}=u_{2,R}\}\cap B_L(0)$ is the graph of a smooth function
$h_R$ in the direction of $\tilde{e}$. By the convergence of
$(u_{i,R})$ and the uniform lower bound on $\inf_{B_L(0)}|\nabla
u_{1,R}-\nabla u_{2,R}|$, this function converges to $0$ in a smooth
way. The conclusion then follows.
\end{proof}

Finally, we prove the two corollaries in Section 1.
\begin{proof}[Proof of Corollary \ref{coro 1}]
 Take an arbitrary point $x_0$. Let
$\rho:=\mbox{dist}(x_0,\{u_{1,\beta}=u_{2,\beta}\})$, which we
assume to be attained at $y_0$. Choose a $k$ so that
$\rho\in[\theta^{k+1},\theta^k)$. (Notations as in the proof of
Lemma \ref{lem 2.7}.) Let
\[\tilde{u}_{i,\beta}(x):=\frac{1}{\rho}u_{i,\beta}(y_0+\rho x).\]

If $\rho\leq K_3\beta^{-1/4}$, \eqref{1.1} follows from \eqref{3.3}
in Lemma \ref{lem 2.9}.

If $\rho\geq K_3\beta^{-1/4}$,  \eqref{1.1} follows from
\eqref{Morrey bound 1} or \eqref{Morrey bound 2} in Lemma \ref{lem
2.7} and standard interior elliptic estimates. (Note that in a
neighborhood of $(x_0-y_0)/\rho$ either $\tilde{u}_{1,\beta}$ or
$\tilde{u}_{2,\beta}$ is very small compared to the other
component.)
\end{proof}

The proof of Corollary \ref{coro 2} is similar.

\noindent {\bf Acknowledgments.} The author's research was partially
supported by NSF of China No. 11301522.

\end{document}